\documentclass[a4paper,12pt]{amsart}
\usepackage[mathscr]{eucal}
\usepackage{amsmath, amsfonts, amssymb, amscd, amsthm}
\usepackage[english]{babel}
\usepackage{mathabx}
\usepackage{tikz-cd}
\usepackage{comment}

\usepackage[colorlinks]{hyperref}
\definecolor{linkblue}{RGB}{1,1,190}
\definecolor{citered}{RGB}{190,1,1}
\hypersetup{
	linkcolor=linkblue,
	urlcolor=linkblue,
	citecolor=citered
}

\theoremstyle{plain}
\newtheorem{theorem}{\bf Theorem}[section]
\newtheorem{proposition}[theorem]{\bf Proposition}
\newtheorem{lemma}[theorem]{\bf Lemma}

\theoremstyle{definition}

\newtheorem{definition}[theorem]{\bf Definition}

\newtheorem{problem}[theorem]{\bf Problem}

\theoremstyle{definition}

\theoremstyle{definition}

\newcommand{\N}{\mathbb N}
\newcommand{\Z}{\mathbb Z}
\newcommand{\R}{\mathbb R}

  \DeclareMathOperator{\size}{size}

\newcommand{\be}{\begin{equation}}
\newcommand{\ee}{\end{equation}}

\newcommand{\ber}{\begin{eqnarray}}
\newcommand{\eer}{\end{eqnarray}}

\numberwithin{equation}{section}

\begin{document}
\title{How iterated sumsets grow under the assumption of small doubling}

\author{Dongchun Han and Qinghai Zhong}

\address{School of Mathematics, Southwest Jiaotong University, Chengdu, 610031, China}
\email{han-qingfeng@163.com}

\address{Institute of Mathematics and Interdisciplinary Sciences, Xidian University, Xi'an 710126, China}
\email{zhong@xidian.edu.cn}


\begin{abstract}
Let $K>1$ be a real number and let $A\subsetneq \mathbb Z$ be a finite and nonempty subset. 
Suppose $|A+A|<K|A|$. A fundamental open problem is to characterize the size of higher iterated sumsets $nA = A+\cdots+A$ ($n$-fold sumset). In 2004, Terry Tao formally posed the problem of determining the optimal growth function governing iterated sumsets of small-doubling sets. In this manuscript, we give a solution to this long-standing open problem  and the corresponding inverse problem for large enough $n$.
\end{abstract}

\maketitle

\section{Introduction}

The growth of iterated sumsets under the small-doubling condition is a central topic in additive combinatorics. For a finite set $A$ of integers, suppose the doubling ratio $\frac{|A+A|}{|A|}$ is bounded by a constant $K$. A fundamental open problem is to characterize the size of higher iterated sumsets $nA = A+\cdots+A$ ($n$-fold sumset). This question lies at the heart of the structural theory of small‑doubling sets and is closely connected to Freiman’s theorem (\cite{Fr73,Ch02}), the Polynomial Freiman–Ruzsa conjecture (\cite{Ru99,Go-Gr-Ma-Ta25}), and a wide range of applications in number theory and combinatorics (\cite{Ta-Vu06,Ru94,Gr-Ru06,Bo-Ch04,Br-Gr-Ta13,Sa12}).

A foundational result in this area is Plünnecke's inequality (\cite{Pl70,Pe11}), which states that if $|2A| \le K|A|$, then for all $n$,
\[
|nA| \le K^n |A|.
\]
This provides an exponential upper bound in $n$, but it is far from optimal for large $n$. The inequality follows from graph-theoretic arguments and gives a first rough estimate of the growth of sumsets.

A major breakthrough came with Freiman's theorem (\cite{Fr73,Ch02,Ta-Vu06}), which provides a structural characterization: if $|2A| \le K|A|$, then $A$ is contained in a generalized arithmetic progression of dimension $O_K(1)$ and size $O_K(|A|)$. This implies that the growth of $nA$ is at most polynomial in $n$, with degree depending on $K$.
However, the constants in Freiman's theorem were initially very poor, and the theorem did not give the optimal polynomial growth. This motivated the search for sharper quantitative bounds.

At the 2004 Palo Alto workshop on additive combinatorics, Terry Tao formally posed the problem of determining the optimal growth function governing iterated sumsets of small-doubling sets. 
\begin{problem}[{\cite[Problem 3.16]{problems} \& \cite[Problem 4.8]{CrootLev2007}}]\label{p1}
	Suppose $A \subsetneq \mathbb{Z}$ is a finite and nonempty subset. If $|2A| < K|A|$, describe the sumset $|nA|$ as $n$ tends to infinity. Find $f(n,K)$, which is the smallest number such that
	\[
	|nA| \leq f(n,K)|A|
	\]
	for all $A$ such that $|2A| < K|A|$.
\end{problem}

If $A$ is fixed, by Hovanskii’s theorem (\cite{Kh92,Kh95, Na00}), we know that there exists a polynomial $f(n)$ such that $|nA|=f(n)$ for all large $n$. 
 Recent work has sharpened several
related aspects of this problem. Eliahou and Mazumdar
\cite{El-Ma22} obtained refined inequalities between
successive sumset cardinalities using Hilbert functions and Macaulay's
theorem, and subsequently established the optimality of these inequalities
over abelian semigroups \cite{El-Ma25}. Their sharpness result
does not establish optimality for subsets of $\mathbb Z$.
Granville, Shakan, and Walker \cite{Gr-Sh-Wa23} obtained
effective thresholds for the eventual polynomial size and structure of
iterated sumsets in $\mathbb Z^d$. For each fixed finite
$A\subsetneq\mathbb Z$ with $|A|\ge2$, the growth is eventually linear,
with $|nA|=\frac{\max A-\min A}{\gcd(A-A)}n+O_A(1)$; this differs from
the uniform extremal problem, in which $A$ may vary with $n$.
More recently, Hegde \cite{He25} proved inverse and stability results
for Ruzsa's triple-sumset inequality and its higher-sumset analogues.
These developments provide sharper cardinality estimates and structural
information, without determining the optimal uniform function $f(n,K)$.

In this manuscript, we consider the case that $K$ is fixed and give a solution to this long-standing open problem for large enough $n$.

\begin{theorem}\label{main1.1}
	Let $K>1$ be a real number and denote 
	\[
	f(n,K)=\sup \left\{\frac{|nA|}{|A|}\colon A\subsetneq \Z \text{ is a finite nonempty subset with }|2A|<K|A|\right\}\,.
	\]
	\begin{enumerate}
		
		\item If $K\le 2$, then for every positive integer $n$, we have
		\[
		f(n,K)=
		\begin{cases}
			1, &\text{ if } K\le \frac{3}{2};\\
			(1-\frac{1}{M})n+\frac{1}{M}, \text{ where }M=\lceil \frac{1}{2-K}\rceil-1, &\text{ if }3/2<K<2;\\
			n, &\text{ if } K=2\,.
		\end{cases}	
		\]
		
		\item If $K>2$, then there exists $N$ such that for every positive integer $n\ge N$, we have 
		\[
		f(n,K)={n+d-1\choose d}-\frac{n-1}{M}{n+d-1\choose d-1}=\frac{M-d}{d!M}n^d+O_K(n^{d-1})\,,
		\]
		where $d=\lceil 2K-3\rceil$ and $M=\left\lceil\frac{d(d+1)}{2(d+1-K)} \right\rceil-1$.
	\end{enumerate}
\end{theorem}

While the above theorem is stated under the strict doubling condition $|A+A|<K|A|$,  the same method can be applied to obtain  an alternative version for the relaxed assumption $|A+A|\le K|A|$.

\begin{theorem}\label{main1.2}
	Let $K>1$ be a real number and denote 
	\[
	\widebar f(n,K)=\sup \left\{\frac{|nA|}{|A|}\colon A\subsetneq \Z \text{ is a finite nonempty subset with }|2A|\le K|A|\right\}\,.
	\]
	\begin{enumerate}
		
		    \item If $K\le 2$, then for every positive integer $n$, we have
		    \[
		    \widebar f(n,K)=
		    \begin{cases}
		    	1, &\text{ if } K< \frac{3}{2};\\
		    	(1-\frac{1}{M})n+\frac{1}{M}, \text{ where }M=\lfloor \frac{1}{2-K}\rfloor, &\text{ if }3/2\le K<2;\\
		    	\frac{(n+1)(n+2)}{6}, &\text{ if }K=2\,.
		    \end{cases}	
		    \]

			\item If $K> 2$, then there exists $N$ such that for every positive integer $n\ge N$, we have 
			\[
			\widebar f(n,K)={n+d-1\choose d}-\frac{n-1}{M}{n+d-1\choose d-1}=\frac{M-d}{d!M}n^d+O_K(n^{d-1})\,,
			\]
			where $d=\lfloor 2K-2\rfloor$ and $M=\left\lfloor\frac{d(d+1)}{2(d+1-K)} \right\rfloor$.
		\end{enumerate}
\end{theorem}

The following theorem investigates the structure of subsets $A$ satisfying $|A+A|<K|A|$ (resp. $|A+A|\le K|A|$) for which the ratio $\frac{|nA|}{|A|}$ is close to $f(n,K)$ (resp. $\overline{f}(n,K)$).
When $K\le 2$, the condition $|A+A|<K|A|$ (resp. $|A+A|\le K|A|$) forces $A$ to be an arithmetic progression, or  a set obtained by removing one element from an arithmetic progression, or $|A|=3$ (see Lemma \ref{special}).
We therefore restrict our attention to the case $K>2$. See Definition \ref{C} for the definition of   the sets $\mathcal C_{d,M}$, and see the beginning of Section \ref{se2} for the definition of Freiman dimension $\dim_F(A)$ of a finite and nonempty subset $A\subsetneq \mathbb Z$.

\begin{theorem}\label{main1.3}
	Let $K> 2$ be a real number. Then there exist $\eta_K>0$ and $N_K$ such that for every $n\ge N_K$, the following hold.
	\begin{enumerate}
		\item Let $d=\lceil 2K-3\rceil$ and $M=\left\lceil\frac{d(d+1)}{2(d+1-K)} \right\rceil-1$.
		If $A\subsetneq \Z$ is a finite and nonempty subset such that 
		\[
		\text{$|A+A|<K|A|$\quad and\quad $\frac{|nA|}{|A|}\ge f(n,K)-\eta_Kn^d$,}
		\] then
		\[
		|A|=M, \quad \dim_F(A)=d,\quad |A+A|=(d+1)|A|-{d+1\choose 2}\,,
		\]
		and there exist $B_n\in \mathcal C_{d,M}$, $a_n\in \Z$, and a homomorphism $\lambda\colon \Z^d\rightarrow \Z$ such that 
		\[
		\text{$A=a_n+\lambda(B_n)$\quad and\quad $\frac{|nB_n|-|\lambda(nB_n)|}{|B_n|}\le \eta_Kn^d$.}
		\]
	
		\item Let $d=\lfloor 2K-2\rfloor$ and $M=\left\lfloor\frac{d(d+1)}{2(d+1-K)} \right\rfloor$. If $A\subsetneq \Z$ is a finite and nonempty subset such that
		\[
		\text{$|A+A|\le K|A|$\quad and\quad $\frac{|nA|}{|A|}\ge \widebar f(n,K)-\eta_Kn^d$,}
		\]  then
		\[
		|A|=M, \quad \dim_F(A)=d,\quad |A+A|=(d+1)|A|-{d+1\choose 2}\,,
		\]
		and there exist $B_n\in \mathcal C_{d,M}$, $a_n\in \Z$, and a homomorphism $\lambda\colon \Z^d\rightarrow \Z$ such that 
		\[
		\text{$A=a_n+\lambda(B_n)$\quad  and\quad  $\frac{|nB_n|-|\lambda(nB_n)|}{|B_n|}\le \eta_Kn^d$.}
		\]
	
	\end{enumerate}
\end{theorem}

In Section 2, we collect some necessary lemmas and prove some properties of subsets $A\subsetneq \Z$ with $|A+A|<K|A|$ having maximal Freiman dimension. In Section 3, we show the structure of subsets $A\subsetneq \Z$ with $(\dim_F(A), |A|)$ suitable and $|2A|$ minimal, which can be applied to both cases $|2A|<K|A|$ and $|2A|\le K|A|$. Our main Theorems will be proved in Section 4. The final section investigates Katz's problem, which is closely related to Problem \ref{p1}.

\section{Preliminaries}\label{se2}

Let $\Z$ be the set of all integers and let $\R$ be the set of all real numbers. For $a,b\in \Z$ with $a\le b$, we denote $[a,b]$ the discrete intervals between $a$ and $b$.

Given a subset $P\subseteq\R^d$, 
the \emph{convex hull} $\operatorname{conv}(P)$ is the smallest convex subset of $\R^d$ containing $P$, i.e.,
\[
\operatorname{conv}(P)=\left\{\,\sum_{i=1}^m \lambda_i s_i \,\colon\,
m\ge 1,\ s_i\in P,\ \lambda_i\ge 0,\ \sum_{i=1}^m\lambda_i=1\,\right\},
\]
and the \emph{affine hull} $\operatorname{aff}(P)$ is the smallest affine subspace containing $P$, i.e.,
\[
\operatorname{aff}(P)=\left\{\,\sum_{i=1}^m \lambda_i p_i \,\colon\,
m\ge 1,\ p_i\in P,\ \lambda_i\in \R,\ \sum_{i=1}^m\lambda_i=1\,\right\}\,.
\]
The dimension of a polytope $P$ is defined as $\dim P=\dim \operatorname{aff}(P)$. For a subset $A\subsetneq \Z^d\subseteq \R^d$, we say $A$ is \emph{$d$-dimensional} if $A$ is not contained in any proper affine subspace of $\R^d$.

Let $A,B$ be two subsets of torsion-free abelian groups. We say a map $\psi\colon A\rightarrow B$ is a \emph{Freiman $2$-homomorphism} if for any $a,b,c,d\in A$, we have 
\[
a+b=c+d  \Rightarrow \psi(a)+\psi(b)=\psi(c)+\psi(d)\,.
\]
If $\psi$ is bijective and its inverse is also a \emph{Freiman $2$-homomorphism}, then we call $\psi$ a Freiman $2$-isomorphism.
The \emph{Freiman dimension} $\dim_F(A)$ is defined as the largest integer $d$ such that $A$ is Freiman $2$-isomorphic to a $d$-dimensional subset of $\Z^d$.

We first collect some simple results.

\begin{lemma}\label{simple}
	Let $A,B\subsetneq \Z$ be two finite subsets with $|A|, |B|\ge 2$. Then $|A+B|\ge |A|+|B|-1$ and equality holds if and only if $A, B$ are arithmetic progressions with common difference.
\end{lemma}
\begin{proof}
See \cite[Proposition 5.8]{Ta-Vu06}.
\end{proof}

\begin{lemma}\label{d=1}
	Let $A$ be a finite and nonempty subset of $\Z$  and let $n\ge 2$ be a positive integer. Then 
	\[
	|nA|\ge n|A|-(n-1)\,.
	\]
	and equality holds if and only if $A$ is an arithmetic progression.
\end{lemma}

\begin{proof}The assertion is trivial when $|A|=1$. Assume that $|A|\ge 2$.
	If $n=2$, then it follows from Lemma \ref{simple} by applying $A=B$. Suppose $n\ge 3$. We proceed by induction. Assume that the assertion holds for $n-1$. By Lemma \ref{simple}, we have
	\[
	|nA|=|(n-1)A+A|\ge |(n-1)A|+|A|-1\ge (n-1)|A|-(n-2)+|A|-1=n|A|-(n-1)\,.
	\]
	
	If $|nA|= n|A|-(n-1)$, then all equalities hold in the above inequality and hence $|(n-1)A|=(n-1)|A|-(n-1)$. By the assumption hypothesis, we obtain that $A$ is an arithmetic progression. 
	
	If $A$ is an arithmetic progression, then it is easy to see that $|nA|= n|A|-(n-1)$.
\end{proof}

The following lemma investigates the growth of iterated sumsets of $d$-dimensional subsets, which plays a key role in our proofs. To state the lemma, we need the following terminology (see \cite{Bo-Sa-Se14} for more detail).

Points $v_0,v_1,\dots,v_k\in\R^d$ are \emph{affinely independent} if the vectors
$v_1-v_0,\,v_2-v_0,\dots,v_k-v_0$ are linearly independent, or
equivalently, the only scalars $\lambda_0,\dots,\lambda_k$ satisfying
\[
\sum_{i=0}^k\lambda_i v_i=0 \quad \text{and}\quad \sum_{i=0}^k\lambda_i=0
\]
are $\lambda_0=\lambda_1=\ldots=\lambda_k=0$.

A \emph{$d$-simplex} is the convex hull
\[
\Delta=\operatorname{conv}\{v_0,v_1,\dots,v_d\}
\]
of $d+1$ affinely independent points $v_0,\dots,v_d\in\mathbb R^d$.
A $k$-face of a simplex is the convex hull of any $(k+1)$‑element subset of its vertices.
In particular, a \emph{$2$-face of a $d$-simplex} is the convex hull of any three vertices of the simplex; it is a triangular $2$-dimensional face.

Given a polytope $P\subsetneq\mathbb R^d$ and a point $w\notin \operatorname{aff}(P)$, the
\emph{pyramid} over base $P$ with apex $w$ is
\[
\operatorname{pyr}(P,w)=\operatorname{conv}(P\cup\{w\}).
\]
We have $\dim \operatorname{pyr}(P,w)=\dim P+1$.
An \emph{iterated pyramid} is obtained by repeated pyramid constructions: start with a polytope $Q$, set $P_0=Q$, and recursively
$P_{i+1}=\operatorname{pyr}(P_i,w_i)$, where each apex $w_i$ lies outside $\operatorname{aff}(P_i)$.

A \emph{simplex-prism} is any polytope affinely equivalent to 
\[
\{(x,t)\in \R^{d+1}\colon x\in \Delta\subsetneq \R^d, t\in \R, \ell(x)\le t\le u(x)\}\,,
\] 
where $\Delta$ is a $d$-simplex and $\ell,u$ are two affine functions with $\ell(x)\le u(x)$ for every $x\in \Delta$.
For each vertex $v$ of $\Delta$, the pair $\big((x,\ell(v)), (x,u(v))\big)$ spans an edge connecting the original base simplex $\Delta$ to its translated copy; these are called the \emph{vertical edges} of the simplex‑prism.

\begin{lemma}\label{d2}
	Let $B\subsetneq \Z^d$ be finite and $d$-dimensional. 
	\begin{enumerate}
		\item For every positive integer $n\ge 2$, we have
		\[
		|nB|\ge {n+d-1\choose d}|B|-{n+d-1\choose d-1}(n-1)\,.
		\]
		
		\item Suppose $d\ge 2$ and $n\ge 2$. Then the equality holds for the above inequality 
		if and only if one of the following holds.
		\begin{enumerate}
			\item The set $B$ consists of the vertices of a simplex plus an arithmetic progression
			contained in an edge of the simplex, starting and ending at the endpoints of the
			edge.

			\item The set $B$ consists of the vertices of a simplex plus the midpoints of the sides of
			a certain $2$-face of the simplex.
			
			\item  For $d \ge q \ge 2$, $\operatorname{conv}(B)$ is an iterated pyramid over a $q$-dimensional simplex-prism.
			There exists a non-zero $w\in \mathbb R^d$ such that $B$ consists of the vertices of $\operatorname{conv}(B)$ and,
			for each vertical edge of the simplex-prism, the arithmetic progression of difference $w$
			starting and ending at its endpoints.
		\end{enumerate}
	\end{enumerate}
\end{lemma}

\begin{proof}
	See \cite[Corollaries 2 and 6]{Bo-Sa-Se14}.
\end{proof}

Next, we give the definition of general progressions, which will be used in the proof of the following lemma.
Let $r$ and  $L_1,\dots,L_r$ be positive integers. Then a \emph{progression} of dimension $r$ is a set of the form \[
P := \left\{ x_0 + \sum_{i=1}^r \mu_i x_i  \colon \mu_i \in [0, L_i]\right\}\,.\]
The \emph{size} (or \emph{volume}) of $P$ is defined by
\[ \size(P) := (L_1+1) \cdot\ldots\cdot (L_r+1).\]
If  $\size(P) = |P|$, then we say  $P$ is \emph{proper}.

\begin{lemma}\label{2.3}
		Let $A$ be a finite and nonempty subset of $\Z$.
		\begin{enumerate}
			\item If $|2A|<K|A|$, then $\dim_F(A)<2K-2$.
			
			\item If $|2A|\le K|A|$, then $\dim_F(A)\le 2K-2$ and 
			\[
			|nA|\le C_Kn^{\dim_F(A)}|A|\,,
			\]
			where $C_K$ depends only on $K$.
		\end{enumerate}
\end{lemma}

\begin{proof}
	Let $\dim_F(A)=d$. We have $|A|\ge d+1$ and there exists a $d$-dimensional subset $B\subsetneq \R^d$ such that $A$ is Freiman $2$-isomorphic to $B$. It follows from Lemma \ref{d2}.1 that 
	\[
	|A+A|=|B+B|\ge (d+1)|B|-{d+1\choose 2}=(d+1)|A|-{d+1\choose 2}\,,
	\] 
	
	whence $$\frac{|A+A|}{|A|}\ge (d+1)-\frac{d(d+1)}{2|A|}\ge d+1-\frac{d}{2}=\frac{d+2}{2}\,.$$

	1. If $|2A|<K|A|$, then $K>\frac{|A+A|}{|A|}\ge \frac{d+2}{2}$ and hence $d<2K-2$.

	2. If $|2A|\le K|A|$, then $K\ge \frac{|A+A|}{|A|}\ge \frac{d+2}{2}$ and hence $d\le 2K-2$. 
	It remains to show there exists a constant $C_K$ only depending on $K$ such that $|nA|\le C_Kn^{d}|A|$.
	
	If $|A|=1$, the assertion is trivial. Suppose $|A|\ge 2$. It follows from \cite[Proposition 1.3]{Gr-Ta06} or \cite{Ch02} that there exists a proper $r$-dimensional progression 
	\[
	P=\left\{ x_0 + \sum_{i=1}^r \mu_i x_i  \colon \mu_i \in [0, L_i]\right\}
	\]
	with $r\le d$ and $|P|=\size(P)=(L_1+1)\cdot\ldots (L_r+1)\le C_K|A|$ such that $A\subsetneq P$, where $C_K$ is a constant depending only on $K$. Then 
	\begin{align*}
	|nA|\le |nP|&=|\{nx_0+\sum_{i=1}^r \mu_i x_i  \colon \mu_i \in [0, nL_i]\}|\\
	&\le (nL_1+1)\cdot\ldots\cdot (nL_r+1)\\
	&\le n^r(L_1+1)\cdot\ldots (L_r+1)\\
	&\le C_Kn^d|A|\,.
	\end{align*}
 \end{proof}

The final lemma in this section investigates the subsets with $|2A|<K|A|$ (resp. $|2A|\le K|A|$) such that $\dim_F(A)$ is maximal.
\begin{lemma}\label{4.1}
	Let $K>2$ and let $A$ be a finite and nonempty subset of $\Z$ with $\dim_F(A)=d\ge 1$ and $|A|=m$.
	\begin{enumerate}
		\item If $|2A|<K|A|$  and $d=\lceil 2K-2 \rceil -1$, then $$|2A|=(d+1)m-{d+1\choose 2} \quad \text{ and }\quad m\le \left\lceil \frac{d(d+1)}{2(d+1-K)} \right\rceil-1\,.$$
		
		\item If $|2A|\le K|A|$  and $d=\lfloor 2K-2 \rfloor$, then  $$|2A|=(d+1)m-{d+1\choose 2}\quad \text{ and }\quad m\le \left\lfloor \frac{d(d+1)}{2(d+1-K)} \right\rfloor\,.$$
	\end{enumerate}
\end{lemma}

\begin{proof}
	Since $\dim_F(A)=d$, we have $m=|A|\ge d+1$ and there exists a $d$-dimensional subset $B\subsetneq \R^d$ such that $A$ is Freiman $2$-isomorphic to $B$. It follows from Lemma \ref{d2}.1 that 
	\begin{equation}\label{e2}
	|A+A|=|B+B|\ge (d+1)|B|-{d+1\choose 2}=(d+1)m-{d+1\choose 2}\,.
	\end{equation}
	
	\medskip
	\noindent{\bf Claim A. }{\it If $m=d+1$, then $|A+A|={ d+2\choose 2}=(d+1)m-{d+1\choose 2}$.}
	\begin{proof}[Proof of Claim A]
		Suppose $m=d+1$.  Then those $d+1$ points are affinely independent, whence all ${ d+2\choose 2}$ unordered sums from $A+A$ are distinct, that is  $$|A+A|={ d+2\choose 2}=(d+1)m-{d+1\choose 2}\,.$$
		\qedhere[End of Claim A]
	\end{proof}

	\medskip
	\noindent{\bf Claim B. }{\it If $m\ge d+2$, then $\frac{d+3}{2}m\le (d+1)m-\frac{d(d+1)}{2}+1$.}
	\begin{proof}[Proof of Claim B]
		Suppose $m\ge d+2$. Then 
		\[
		\frac{d(d+1)}{2}-1=\frac{(d+2)(d-1)}{2}\le \frac{m(d-1)}{2}=(d+1)m-\frac{d+3}{2}m\,,
		\]
		whence $\frac{d+3}{2}m\le (d+1)m-\frac{d(d+1)}{2}+1$.
		\qedhere[End of Claim B]
	\end{proof}
	
	1. Assume that  $|2A|<K|A|$  and $d=\lceil 2K-2 \rceil -1$. Then $\frac{d+2}{2}<K\le \frac{d+3}{2}$ and the inequality $|2A|<K|A|$ together with \eqref{e2} implies that $$m\le \left\lceil \frac{d(d+1)}{2(d+1-K)} \right\rceil-1\,.$$
	If $m=d+1$, then  $|A+A|=(d+1)m-{d+1\choose 2}$ follows from Claim A.
	Suppose $m\ge d+2$.  Since $K\le \frac{d+3}{2}$, it follows from Claim B and \eqref{e2} that
	\[
	(d+1)m-{d+1\choose 2}\le |A+A|<K|A|\le \frac{d+3}{2}m\le (d+1)m-\frac{d(d+1)}{2}+1\,.
	\]
	In view of the fact that $|A+A|$ is an integer, we obtain $$|A+A|=(d+1)m-\frac{d(d+1)}{2}\,.$$
	
	2. Assume that $|2A|\le K|A|$  with $d=\lfloor 2K-2 \rfloor$. Then $\frac{d+2}{2}\le K< \frac{d+3}{2}$ and the inequality $|2A|\le K|A|$ together with \eqref{e2} implies that $$m\le \left\lfloor \frac{d(d+1)}{2(d+1-K)} \right\rfloor\,.$$
	If $m=d+1$, then  $|A+A|=(d+1)m-{d+1\choose 2}$ follows from Claim A.
	Suppose $m\ge d+2$.  Since $K<\frac{d+3}{2}$, it follows from Claim B and \eqref{e2} that 
	\[
	(d+1)m-{d+1\choose 2}\le |A+A|\le K|A|< \frac{d+3}{2}m\le (d+1)m-\frac{d(d+1)}{2}+1\,.
	\]
	In view of the fact that $|A+A|$ is an integer, we obtain $|A+A|=(d+1)m-\frac{d(d+1)}{2}$. 
\end{proof}

\section{The subset $A$ with $(\dim_F(A), |A|)$ suitable}\label{new}

\begin{definition}
We say a pair $(d,m)\in \N^2$ is suitable if one of the following holds.
\begin{itemize}

	\item $d\ge 2$ and either $m=d+1$ or $m=d+2$;
	
	\item $d=2$ and $m=5$.
\end{itemize}
\end{definition}

The following lemma shows that the subsets with $|2A|<K|A|$ (resp. $|2A|\le K|A|$) such that $\dim_F(A)$ is maximal has the property that $(\dim_F(A), |A|)$ is suitable (in view of Lemma \ref{4.1}).
\begin{lemma}\label{suitable}
		Let $K>2$ be real.
	\begin{enumerate}
		\item If $d=\lceil 2K-2 \rceil -1$ and $d+1\le m\le \left\lceil \frac{d(d+1)}{2(d+1-K)} \right\rceil-1$, then the pair $(d,m)$ is suitable.
		
		\item If  $d=\lfloor 2K-2 \rfloor$ and $d+1\le m\le \left\lfloor \frac{d(d+1)}{2(d+1-K)} \right\rfloor$, then the pair $(d,m)$ is suitable.
	\end{enumerate}
\end{lemma}

\begin{proof}
1. 	Since $K>2$ and $d=\lceil 2K-2 \rceil -1$, we have $d\ge 2$ and $K\le (d+3)/2$,  whence
$$m\le \left\lceil \frac{d(d+1)}{2(d+1-K)} \right\rceil-1< \frac{d(d+1)}{2(d+1-K)}\le \frac{d(d+1)}{d-1}\,.$$

 If $d=2$, then $m<6$ and hence $(d,m)$ is suitable.
If $d\ge 3$, then $$m<\frac{d(d+1)}{d-1}=d+2+\frac{2}{d-1}\le d+3$$ and hence $(d,m)$ is suitable.

2. 	Since $K>2$ and $d=\lfloor 2K-2 \rfloor$, we have $d\ge 2$ and  $K< (d+3)/2$, whence
$$m\le \left\lfloor \frac{d(d+1)}{2(d+1-K)} \right\rfloor\le \frac{d(d+1)}{2(d+1-K)}<\frac{d(d+1)}{d-1}\,.$$ If $d=2$, then $m<6$ and hence $(d,m)$ is suitable.
If $d\ge 3$, then $$m<\frac{d(d+1)}{d-1}=d+2+\frac{2}{d-1}\le d+3$$ and hence $(d,m)$ is suitable.
\end{proof}

\begin{definition}\label{C}
	For each suitable pair $(d,m)$, we define a set $\mathcal C_{d,m}$, whose elements are $d$-dimensional subsets with size $m$ of $\Z^d$. Let $(e_1,\ldots,e_d)$ be a basis of $\Z^d$ and set
	\begin{itemize}
		
		\item $C_{d,d+1}=\{0, e_1,\ldots,e_d\}$ and $\mathcal C_{d,d+1}=\{C_{d,d+1}\}$, when $d\ge 2$;
		
		\item  $C_{d,d+2}=\{0, e_1, 2e_1,e_2,\ldots,e_d\}$, $C_{d,d+2}^*=\{0, e_1,e_2,e_1+e_2,e_3,\ldots,e_d\}$, \\and $\mathcal C_{d,d+2}=\{C_{d,d+2}, C_{d,d+2}^*\}$, when $d\ge 2$;
		
		\item $C_{2,5}=\{0,e_1,2e_1,3e_1,e_2\}$, $C_{2,5}^*=\{0,e_1,2e_1,e_2,e_1+e_2\}$, and $\mathcal C_{2,5}=\{C_{2,5}, C_{2,5}^*\}$. 
	\end{itemize}
\end{definition}

Those new defined subsets have the following nice property.
\begin{lemma}\label{le3.2}
Suppose $(d,m)$ is suitable and let $B\in \mathcal C_{d,m}$. For every $n\ge 1$, we have 
\[
|nB|={n+d-1\choose d}|B|-{n+d-1\choose d-1}(n-1)\,.
\]
\end{lemma}
\begin{proof}
If $n=1$, then it is trivial. Suppose $n\ge 2$. Since $(d,m)$ is suitable, we have $d\ge 2$. It is easy to check that $B$ satisfies one of the cases of Lemma \ref{d2}.2 and hence the assertion follows.
\end{proof}

%
%
%

\begin{proposition}\label{3.4}
	Let $A\subsetneq \Z$ be a finite and nonempty subset with $\dim_F(A)=d$ and $|A|=m$ such that the pair $(d,m)$ is suitable and  $|2A|=(d+1)m-{d+1\choose 2}$.
	\begin{enumerate}
		\item 	There exist $B\in \mathcal C_{d,m}$, 	$a\in \Z$, and a homomorphism $\lambda\colon \Z^d\rightarrow \Z$ such that 
		\[
		A=a+\lambda(B)\,.
		\]
		\item For every positive integer $n$, we have 
		\[
		|nA|=|\lambda(nB)|\le |nB|={n+d-1\choose d}|A|-{n+d-1\choose d-1}(n-1)\,.
		\]
		
		\item Equality holds in the above inequality if and only if $\lambda$ is injective on $nB$, or equivalently, $\ker \lambda\cap (nB-nB)=\{0\}$. 
	\end{enumerate}

\end{proposition}

\begin{proof}
	Since $\dim_F(A)=d$, there exists a $d$-dimensional subset $C\subsetneq \Z^d$ such that $A$ is Freiman $2$-isomorphic to  $C$. Then $|C|=|A|=m$ and $|2C|=|2A|=(d+1)m-{d+1\choose 2}$. Since $(d,m)$ is suitable, the set $C_{d,m}$ is well-defined and $d\ge 2$. We proceed by the following claim.
	
	\medskip
	\noindent{\bf Claim A. }{\it There exists $B\in C_{d,m}$ such that  $C$ is Freiman $2$-isomorphic to $B$.}
	\medskip
	
	\begin{proof}[Proof of Claim A]
		Note that $d\ge 2$. We apply Lemma \ref{d2}.2 for $n=2$ and obtain that $A$ has one of the following forms.

		\begin{enumerate}
			\item[(a)]The set of the vertices of a simplex together with an arithmetic progression on one edge, including the
			endpoints;

			\item[(b)]The set of the vertices of a simplex together with the three edge midpoints of a two-dimensional face;
			
			\item[(c)] The set of the vertices of  an iterated pyramid over a $q$-dimensional simplex-prism, together with arithmetic progressions with one common difference on all vertical edges, including the
			endpoints, where $d\ge q\ge 2$.
		\end{enumerate}
		
		\medskip
	Since $(d,m)$ is suitable, we may distinguish three cases.
		
		If $m = d + 1$, then $A$ is a $d$-dimensional simplex and hence Freiman $2$-isomorphic to  $C_{d,d+1}$.
		
		Suppose $m = d + 2$. In case (a), we have $A$ is the set of the vertices of a simplex together with  exactly one additional point on one edge. Since the
		points on that edge form an arithmetic progression including both endpoints, the additional point is the midpoint, implying $A$ is Freiman $2$-isomorphic to  $C_{d,d+2}$. Case (b) has $d + 4$ points and is impossible. In case (c), a $q$-dimensional simplex-prism has $2q$ vertices, and the iterated pyramid contributes $d - q$ new vertices. Thus
		there are at least $d + q$ points. Since  $m=d + 2\ge q+d$ forces $q = 2$ and forbids interior points on either vertical edge, and since the vertical progressions have a common difference, the $2$-dimensional simplex-prism is a
		parallelogram, whence  after adding the pyramid vertices, $A$ is Freiman $2$-isomorphic to $C_{d,d+2}^*$.
		
		Suppose $(d, m) = (2, 5)$. Case (a) gives four equally spaced points on one edge of a triangle
		and the opposite vertex, implying $A$ is Freiman $2$-isomorphic to  $C_{2,5}$. Case (b) has six points, which is impossible. In case (c), the inequality $d\ge q\ge 2$ implies that $\operatorname{conv}(A)$ is a $2$-dimensional simplex-prism, and hence  $A$ is the union of  two vertical arithmetic progressions with common difference, implying $A$ is Freiman $2$-isomorphic to $C_{2,5}^*$.
		
		\qedhere[End of Claim A]
	\end{proof}

	\medskip
	\noindent{\bf Claim B. }{\it There exist $B\in C_{d,m}$, $a\in \Z$, and a homomorphism $\lambda\colon \Z^d\rightarrow \Z$ such that   $nA=na+\lambda(nB)$ for every positive integer $n$.}
	\medskip
	
	\begin{proof}[Proof of Claim B]
		Since $A$ is Freiman $2$-isomorphic to  $C$, it follows from Claim A that $A$ is Freiman $2$-isomorphic to some $B\in C_{d,m}$. Let $\psi$ be the Freiman $2$-isomorphism from $B$ to $A$.
		
		Note that $(e_1,\ldots,e_d)$ is a basis of $\Z^d$. In view of the structure of $B$, we set $a=\psi(0)$ and set $\lambda\colon \Z^d\rightarrow \Z$ be the only homomorphism satisfying $\lambda(e_i)=\psi(e_i)-a$ for every $i\in \{1,\ldots,d\}$. We want to show $\psi(x)=a+\lambda(x)$ for every $x\in B$. 
		
		Since $\psi$ is a Freiman $2$-isomorphism, it follows from $0+(e_1+e_2)=e_1+e_2$ that 
		\begin{align*}
			\psi(e_1+e_2)&=\psi(e_1)+\psi(e_2)-\psi(0)=\lambda(e_1)+\lambda(e_2)+a=\lambda(e_1+e_2)+a
		\end{align*} 
		and it follows from $(j-1)e_1+e_1=je_1+0$  that 
		\begin{align*}
		\psi(je_1)&=\psi((j-1)e_1)+\psi(e_1)-\psi(0)=\psi((j-1)e_1)+\lambda(e_1) \text{ for every }j\ge 1\,,
		\end{align*}
		whence  $\psi(je_1)=\psi(0)+j\lambda(e_1)=a+\lambda(je_1)$ for every $j\ge 1$. 
		
		In view of the structure of $B$  again, we obtain that  $\psi(x)=a+\lambda(x)$ for every $x\in B$ and hence $A=\psi(B)=a+\lambda(B)$. Since $\lambda$ is a homomorphism, it follows that $$nA=na+n\lambda(B)=na+\lambda(nB)$$ for every $n\ge 1$. 
		\qedhere[End of Claim B]
	\end{proof}
	
 By Claim B, there exist $B\in C_{d,m}$, $a\in \Z$, and a homomorphism $\lambda\colon \Z^d\rightarrow \Z$ such that   $A=a+\lambda(B)$ and hence 
 \[
 nA=na+\lambda(nB) \text{ for every }n\ge 2\,.
 \]
 It follows from Lemma \ref{le3.2} that
	 $$|nA|=|\lambda(nB)|\le |nB|={n+d-1\choose d}|B|-{n+d-1\choose d-1}(n-1)\,.$$
	 Moreover,	$|\lambda(nB)|= |nB|$ if and only if $\lambda$ is injective on $nB$, or equivalently, $$\ker \lambda\cap (nB-nB)=\{0\}\,. $$	 
\end{proof}

\section{Proofs of main Theorems}

\begin{lemma}\label{injective}
	Let $S\subsetneq \Z^d$ be a finite and nonempty subset and let $C$ be a constant. Suppose for every $x=(x_1,\ldots,x_d)\in S$, we have $|x_i|\le C$  for each $i$. If the integer $R> 2C+1$, then the map $\lambda_R$ defined by
	\[
	\lambda_R(x)=x_1+Rx_2+\ldots+R^{d-1}x_d
	\]
	is injective on $S$.
\end{lemma}
\begin{proof}
	Assume to the contrary that $\lambda_R$ is not injective on $S$. Then there exist $x,y\in S$ with $x\neq y$ such that $\lambda_R(x)=\lambda_R(y)$. Set $z=(z_1,\ldots, z_d)=x-y\neq 0$. Then $|z_i|\le 2C$ for each $i$. Let $j\in [1,d]$ be maximal such that $z_j\neq 0$. It follows from $\lambda_R(x)-\lambda_R(y)=\lambda_R(z)=0$ that $j\ge 2$ and
	\begin{align*}
	R^{j-1}&\le |z_j|R^{j-1}=|z_1+Rz_2+\ldots+R^{j-1}z_{j-2}|\\
	&\le 2C(1+R+\ldots+R^{j-2})<\frac{2C}{R-1}R^{j-1}<R^{j-1}\,,
	\end{align*}
	a contradiction.
\end{proof}

Now we are ready to prove our main theorems.
\begin{proof}[Proof of Theorem \ref{main1.1}]
	1.  Suppose $K\le 2$.  Let $A\subsetneq \Z$ be a finite and nonempty subset with $|2A|<K|A|\le 2|A|$. It follows from  Lemma \ref{d=1}  that $|2A|=2|A|-1$, whence $A$ is an arithmetic progression and $$|nA|=n|A|-(n-1)\,,.$$ Furthermore, $2|A|-1=|2A|<K|A|$ implies that 
	$|A|<\frac{1}{2-K}$ and hence 
	\[
	|A|\le \left\lceil \frac{1}{2-K} \right\rceil-1\,, \text{ when } K<2\,.
	\]

	If $K\le 3/2$, then $|A|=1$ and hence $f(n, K)=\sup \left\{\frac{|n\{a\}|}{|\{a\}|}\colon a\in \Z\right\}=1$ for all $n$.
	
	If $3/2<K<2$, then $|A|\le M$, where $M=\left\lceil \frac{1}{2-K} \right\rceil-1$, and 
	\[
	f(n,K)=\sup\left\{\frac{n|A|-(n-1)}{|A|}\colon |A|\le M\right\}=(1-1/M)n+1/M \text{ for all }n\,.
	\]
	
	If $K=2$, then $f(n,K)=\lim_{|A|\rightarrow \infty}\frac{n|A|-(n-1)}{|A|}=n$ for all $n$.

	\medskip
    
    2. Suppose $K>2$. Let $A\subsetneq \Z$ be a finite and nonempty subset with $|2A|<K|A|$ and let $d=\lceil 2K-3\rceil$. Then Lemma \ref{2.3}.1 implies that $\dim_F(A)\le d$. If $\dim_F(A)\le d-1$, then Lemma \ref{2.3}.2 implies that $$|nA|\le C_Kn^{d-1}|A|$$ for some constant $C_K$ depending only on $K$. 
    If $\dim_F(A)=d$, then Lemmas \ref{4.1}.1 and \ref{suitable}.1 imply that the pair $(d, |A|)$ is suitable,
    \[
    |A|\le M\ \text{ and }\ |2A|=(d+1)|A|-{d+1\choose 2}\,, \text{where $M=\left\lceil\frac{d(d+1)}{2(d+1-K)} \right\rceil-1$.}
    \] 
    It follows from Proposition \ref{3.4}.2 that $|nA|\le {n+d-1\choose d}|A|-{n+d-1\choose d-1}(n-1)$ and hence 
    \begin{align*}
    \frac{|nA|}{|A|}&\le {n+d-1\choose d}-{n+d-1\choose d-1}\frac{n-1}{|A|}\\
    &\le {n+d-1\choose d}-{n+d-1\choose d-1}\frac{n-1}{M}\\
    &=\frac{M-d}{d!M}n^d+O_K(n^{d-1})\,.
    \end{align*}
     Since $C_Kn^{d-1}\le \frac{M-d}{d!M}n^d+O_K(n^{d-1})$ for large enough $n$, we obtain that $$f(n,K)\le  {n+d-1\choose d}-{n+d-1\choose d-1}\frac{n-1}{M} \quad \text{for large enough $n$.}$$ 
    
    By Lemma \ref{suitable}, we have that $(d, M)$ is suitable and hence Lemma \ref{le3.2} implies that there exists $B\in \mathcal C_{d,M}$ with $|B|=M$ such that 
    $$|nB|={n+d-1\choose d}|B|-{n+d-1\choose d-1}(n-1) \text{ for all }n\,.$$
    It follows that $|2B|=(d+1)|B|-{d+1\choose 2}<K|B|$ and $$\frac{|nB|}{|B|}={n+d-1\choose d}-{n+d-1\choose d-1}\frac{n-1}{M}\,.$$ By Lemma \ref{injective}, there exists a homomorphism $\lambda\colon \Z^d\rightarrow \Z$ such that $\lambda$ is injective on $B\cup 2B\cup nB$. Let $A_0=\lambda(B)\subsetneq \Z$. Then 
 $$|2A_0|=|2\lambda(B)|=|2B|=(d+1)|B|-{d+1\choose 2}<K|B|=K|A_0|$$ and hence $$f(n,K)\ge \frac{|nA_0|}{|A_0|}=\frac{|nB|}{|B|}={n+d-1\choose d}-{n+d-1\choose d-1}\frac{n-1}{M}\,,$$
     whence 
     \[
     \text{$f(n,K)= {n+d-1\choose d}-{n+d-1\choose d-1}\frac{n-1}{M}$\quad  for large enough $n$. }
     \]
\end{proof}

We first prove a special case of Theorem \ref{main1.2}.

\begin{lemma}\label{special}
	Let 
	\[
	\widebar f(n,2)=\sup \left\{\frac{|nA|}{|A|}\colon A\subsetneq \Z \text{ is a finite nonempty subset with }|2A|\le 2|A|\right\}\,.
	\]
 Then $\widebar f(n,2)=\frac{(n+1)(n+2)}{6}$ for all $n$ and if $A\subsetneq \Z$ with $|2A|\le 2|A|$  such that  $|n_0A|/|A|> n_0$ for some $n_0\ge 3$, then $|A|=3$.

\end{lemma}

\begin{proof}
Let $A\subsetneq \Z$ be a finite and nonempty subset with $|2A|\le 2|A|$. If $|A|\ge 4$, then Freiman's $3k-4$ Theorem (see \cite{Na96} or \cite{Gr13}) implies that $A$ is obtained by deleting one term close to one of the ends from an arithmetic progression, whence $|nA|\le n|A|$ and \begin{equation}\label{s1}
\frac{|nA|}{|A|}\le n\le \frac{(n+1)(n+2)}{6} \text{ when }|A|\ge 4
\end{equation} for all $n$. Since the number of unordered sums from  $nA$ is at most ${n+|A|-1\choose n}$, it is easy to see that
\begin{equation}\label{s2}
\frac{|nA|}{|A|}\le \frac{1}{|A|}{|A|+n-1\choose n}\le \frac{(n+1)(n+2)}{6} \text{ when } |A|\le 3\,.
\end{equation}
 Put all together, we obtain that 
\[
\widebar f(n,2)\le \frac{(n+1)(n+2)}{6}\,.
\]
Let $B\in \mathcal C_{2,3}$. Then $|B|=3$ and $|nB|/|B|=\frac{(n+1)(n+2)}{6}$ (in view of Lemma \ref{le3.2}).
 By Lemma \ref{injective},
there exists a homomorphism $\lambda\colon \Z^2\rightarrow \Z$ such that $\lambda$ is injective on $B\cup 2B\cup nB$. Let $A_0=\lambda(B)\subsetneq \Z$. Then 
$$|2A_0|=|2\lambda(B)|=|2B|=6< 3|A_0|$$ and hence $$\widebar f(n,2)\ge \frac{|nA_0|}{|A_0|}=\frac{|nB|}{|B|}=\frac{(n+1)(n+2)}{6}\,.$$
Therefore 
\[
\text{$\widebar f(n,2)= \frac{(n+1)(n+2)}{6}$\quad  for all $n$. }
\]

If $A\subsetneq\Z$ with $|A+A|\le 2|A|$	and $|n_0A|>n_0|A|$ for some $n_0\ge 3$, then \eqref{s1} and \eqref{s2} imply that $|A|=3$. 
\end{proof}

\bigskip
\begin{proof}[Proof of Theorem \ref{main1.2}]
	1. If $K=2$, then the assertion follows from Lemma \ref{special}.
	Suppose $K< 2$. Let  $A\subsetneq \Z$ be a finite and nonempty subset with $|2A|\le K|A|< 2|A|$. Then Lemma \ref{d=1} implies $2|A|-1\le |2A|\le K|A|<2|A|$, whence 
	\[
	    |2A|=2|A|-1 \text{ and } |A|\le \left\lfloor \frac{1}{2-K}\right\rfloor\,.
	\]
	By Lemma \ref{d=1} again, we have $A$ is an arithmetic progression and hence \[
	\text{$|nA|=n|A|-(n-1)$ for all $n$.}
	\]
	
	If $K<3/2$, then $|A|\le \left\lfloor \frac{1}{2-K}\right\rfloor=1$, whence $|A|=1$. It follows that 
	\[
	\text{$\widebar f(n,K)=1$ for all $n$ when $K<3/2$. }
	\] 
	
	If $3/2\le K<2$, then $\frac{|nA|}{|A|}=n-\frac{n-1}{|A|}$ with $|A|\le M$, where $M=\left\lfloor \frac{1}{2-K}\right\rfloor$, whence 	
	\begin{align*}
	\widebar f(n,K)&=\sup\left\{\frac{n|A|-(n-1)}{|A|}\colon |A|\le M\right\}\\
	&=(1-1/M)n+1/M \qquad \text{ for all }n\,.
	\end{align*}

	2.  Suppose $K> 2$. Let $A\subsetneq \Z$ be a finite and nonempty subset with $|2A|\le K|A|$, and let
	\[
	\text{ $d=\lfloor 2K-2\rfloor$ \quad  and \quad  $M=\left\lfloor\frac{d(d+1)}{2(d+1-K)} \right\rfloor$.}
	\]
	 Then Lemma \ref{2.3}.2 implies that $\dim_F(A)\le d$ and if $\dim_F(A)\le d-1$, then $$|nA|\le C_Kn^{d-1}|A|$$ for some constant $C_K$ depending only on $K$.

	If $\dim_F(A)=d$, then Lemmas \ref{4.1}.2 and \ref{suitable}.2 imply that the pair $(d, |A|)$ is suitable,
	\[
	|A|\le M\quad \text{ and }\quad |2A|=(d+1)|A|-{d+1\choose 2}.
	\] 
	It follows from Proposition \ref{3.4}.2 that $|nA|\le {n+d-1\choose d}|A|-{n+d-1\choose d-1}(n-1)$ and hence 
	\begin{align*}
		\frac{|nA|}{|A|}&\le {n+d-1\choose d}-{n+d-1\choose d-1}\frac{n-1}{|A|}\\
		&\le {n+d-1\choose d}-{n+d-1\choose d-1}\frac{n-1}{M}\\
		&=\frac{M-d}{d!M}n^d+O_K(n^{d-1})\,,
	\end{align*}
 Since $C_Kn^{d-1}\le \frac{M-d}{d!M}n^d+O_K(n^{d-1})$ for large enough $n$, we obtain that $$\widebar{f}(n,K)\le  {n+d-1\choose d}-{n+d-1\choose d-1}\frac{n-1}{M} \quad \text{for large enough $n$.} $$
	
	In view of Lemma \ref{suitable}, we have $(d, M)$ is suitable and hence Lemma \ref{le3.2} implies that there exists $B\in \mathcal  C_{d,M}$ with $|B|=M$ such that 
	$$|nB|={n+d-1\choose d}|B|-{n+d-1\choose d-1}(n-1)\,.$$
	It follows that $|2B|=(d+1)|B|-{d+1\choose 2}\le K|B|$ and $$\frac{|nB|}{|B|}={n+d-1\choose d}-{n+d-1\choose d-1}\frac{n-1}{M}\,.$$ By Lemma \ref{injective}, there exists a homomorphism $\lambda\colon \Z^d\rightarrow \Z$ such that $\lambda$ is injective on $B\cup 2B\cup nB$. Let $A_0=\lambda(B)\subsetneq \Z$. Then 
	$$|2A_0|=|2\lambda(B)|=|2B|=(d+1)|B|-{d+1\choose 2}\le K|B|=K|A_0|$$ and hence $$\widebar f(n,K)\ge \frac{|nA_0|}{|A_0|}=\frac{|nB|}{|B|}={n+d-1\choose d}-{n+d-1\choose d-1}\frac{n-1}{M}\,.$$
	Therefore 
	\[
	\text{$\widebar f(n,K)= {n+d-1\choose d}-{n+d-1\choose d-1}\frac{n-1}{M}$\quad  for large enough $n$. }
	\]
\end{proof}

\begin{proof}[Proof of Theorem \ref{main1.3}] Since the proofs for Items 1 and 2 are similar, we only prove the first one.
	
	Let  $C_K$ be the constant defined in Lemma \ref{2.3} and let
	\[
	C_1^*=\frac{M-d}{d!M}\quad \text{ and }\quad C_2^*=
		\frac{1}{(d-1)!M(M-1)}\,.
	\] 
	Recall $f(n,K)=\frac{M-d}{d!M}n^d+O_K(n^{d-1})$ in Theorem \ref{main1.1}.2.
	Choose $\eta_K<\frac{1}{2}\min\{C_1^*, C_2^*\}$. Then $$f(n,K)-\eta_Kn^d=\left(\frac{M-d}{d!M}-\eta_K\right)n^d+O_K(n^{d-1})\,.$$ Since $\frac{M-d}{d!M}-\eta_K>0$, there exists $N_K^*$ such that for all $n\ge N_K^*$, we have 
	\begin{equation}\label{N_K^*}
	f(n,K)-\eta_Kn^d\ge C_Kn^{d-1}\,.
	\end{equation}
	Since $${n+d-1\choose d-1}\frac{n-1}{M(M-1)}=\frac{1}{(d-1)!M(M-1)}n^d+O_K(n^{d-1})$$ and $\eta_K<\frac{1}{(d-1)!M(M-1)}$, there exists $N_K'$ such that for all $n\ge N_K'$, we have 
	\begin{equation}\label{N_K'}
	\eta_Kn^d<{n+d-1\choose d-1}\frac{n-1}{M(M-1)}\,.
	\end{equation}
	 Choose $N_K=\max\{N_K^*, N_K'\}$. 
	 Let $A\subsetneq \Z$ be a finite and nonempty subset such that
	 \[
	 \text{ $|A+A|<K|A|$\quad  and\quad  $\frac{|nA|}{|A|}\ge f(n,K)-\eta_Kn^d$\quad  for some $n\ge N_K$. }
	 \]
	  If $\dim_F(A)\le d-1$, then Lemma \ref{2.3}.2 and \eqref{N_K^*} imply that $$|nA|\le C_Kn^{d-1}|A|\le f(n,K)|A|-\eta_Kn^{d}|A|\,,$$
	 a contradiction to the assumption of $\frac{|nA|}{|A|}$. By Lemma  \ref{2.3}.1, we obtain $$\dim_F(A)=d\,.$$
	In view of Lemma \ref{4.1}.1, we have $$|2A|=(d+1)|A|-{d+1\choose 2}$$ and $|A|\le M$.
	Then Proposition \ref{3.4}.2 implies that $$|nA|\le {n+d-1\choose d}|A|-{n+d-1\choose d-1}(n-1)\,.$$
	If $|A|\le M-1$, then 
	\begin{align*}
	\eta_Kn^d\ge f(n,K)-\frac{|nA|}{|A|}&\ge {n+d-1\choose d-1}\frac{(n-1)(M-|A|)}{|A|M}\\
	&\ge {n+d-1\choose d-1}\frac{n-1}{M(M-1)}\,,
	\end{align*}
a contradiction to \eqref{N_K'}, whence $$|A|=M\,.$$ In view of Lemma \ref{suitable}, the pair $(d,M)$ is suitable and it follows from Proposition \ref{3.4}.1 that there exists  $B_n\in \mathcal C_{d,m}$, 	$a_n\in \Z$, and a homomorphism $\lambda\colon \Z^d\rightarrow \Z$ such that 
\[
A=a_n+\lambda(B_n)\,,
\]
whence 
\[
\eta_Kn^d\ge f(n,K)-\frac{|nA|}{|A|}\ge \frac{|nB_n|}{|B_n|}-\frac{|\lambda(nB_n)|}{|B_n|}\,.
\]
\end{proof}

\section{Katz's problem}

One problem closely related to Problem \ref{p1} is Katz’s Problem(\cite[Problem 3.17]{problems}), which asks whether there is a finite subset $A\subsetneq\Z$ with $|2A|=K|A|$ such that every proper subset $A_0\subsetneq A$ satisfies $|nA_0|\ge f(n,K)|A_0|$, where $f(n,K)$ is defined in Theorem \ref{main1.1}. The following result gives an answer to this problem.
\begin{theorem}
	Let $K>1$ be a real number and let $n\ge 2$ be an integer. Suppose $A\subsetneq \Z$ be a finite and nonempty subset with $|A+A|=K|A|$ such that every nonempty subset $A_0\subsetneq A$ satisfies $|nA_0|\ge f(n,K)|A_0|$, where $f(n,K)$ is defined in Theorem \ref{main1.1}. Then $K=3/2$ and $|A|=2$.
\end{theorem}

\begin{proof}
	Since $K>1$, it follows from $|A+A|=K|A|>|A|$ that $|A|\ge 2$. Let $a\in A$. Then $\{a\}\subsetneq A$ and hence   $$1=|n\{a\}|\ge f(n,K)|\{a\}|=f(n,K)\,.$$

	Suppose $K>3/2$. Let $B=\{0,1\}$. Then $|B+B|=\frac{3}{2}|B|<K|B|$ and hence $f(n,K)\ge \frac{|nB|}{|B|}=\frac{n+1}{2}>1$, a contradiction.

	Suppose $K\le 3/2$. In view of Lemma \ref{d=1}, we have $K|A|=|2A|\ge 2|A|-1$ and hence 
	\[
	3/2\ge K\ge 2-\frac{1}{|A|}\ge 3/2\,. 
	\]
	It follows that all equalities hold, whence $K=3/2$ and $|A|=2$.
\end{proof}

\section*{Acknowledgement}
 This work was supported by the National Natural Science Foundation of China (Grant No.~12571578) and by Xidian University (Grant No.~21103260003).

\providecommand{\bysame}{\leavevmode\hbox to3em{\hrulefill}\thinspace}
\providecommand{\MR}{\relax\ifhmode\unskip\space\fi MR }
\providecommand{\MRhref}[2]{%
	\href{http://www.ams.org/mathscinet-getitem?mr=#1}{#2}
}


\begin{thebibliography}{1}





\bibitem{Bo-Sa-Se14}K.~J.~ Böröczky, F.~Santos, O.~Serra,\emph{On sumsets and convex hull}, Discrete \& Computational Geometry, \textbf{52}(2014), 705-729.

\bibitem{Bo-Ch04}
J. Bourgain and M.-C. Chang, \emph{On the size of $k$-fold sum and product sets of integers},
J. Amer. Math. Soc. \textbf{17}(2004), no. 2, 473--497.

\bibitem{Br-Gr-Ta13}
E. Breuillard, B. Green, and T. Tao, \emph{Small doubling in groups},
in Erd\H{o}s Centennial, Bolyai Soc. Math. Stud. \textbf{25},
J\'anos Bolyai Math. Soc., Budapest, 2013, 129--151.

\bibitem{Ch02}
M.-C. Chang, \emph{A polynomial bound in Freiman's theorem},
Duke Math. J. \textbf{113}(2002), no. 3, 399--419.

\bibitem{problems}E.~Croot and V.~Lev (collectors), \emph{Problems presented at the workshop on Recent Trends in Additive Combinatorics}, Palo Alto Workshop, 2004. Available at \url{https://aimath.org/WWN/additivecomb/additivecomb.pdf}.

\bibitem{CrootLev2007}
E.~Croot and V.~F.~Lev,
\emph{Open problems in additive combinatorics},
in \emph{Additive Combinatorics}, CRM Proceedings \& Lecture Notes
\textbf{43}, American Mathematical Society, Providence, RI, 2007,
207--233.


	
\bibitem{El-Ma22}
S.~Eliahou and E.~Mazumdar,
\emph{Iterated sumsets and Hilbert functions},
J. Algebra \textbf{593} (2022), 274--294.


\bibitem{El-Ma25}
S.~Eliahou and E.~Mazumdar,
\emph{Optimal bounds on the growth of iterated sumsets in abelian semigroups},
Ann. Inst. Fourier (Grenoble) \textbf{75} (2025), no.~6, 2321--2339.

\bibitem{Fr73}
G. A. Freiman, \emph{Foundations of a Structural Theory of Set Addition},
Transl. Math. Monogr. \textbf{37}, Amer. Math. Soc., Providence, RI, 1973.



\bibitem{Go-Gr-Ma-Ta25}
W. T. Gowers, B. Green, F. Manners, and T. Tao, \emph{On a conjecture of Marton},
Ann. of Math. (2) \textbf{201}(2025), no. 2, 515--549.

\bibitem{Gr-Sh-Wa23}
A.~Granville, G.~Shakan, and A.~Walker,
\emph{Effective results on the size and structure of sumsets},
Combinatorica \textbf{43} (2023), 1139--1178.


\bibitem{Gr-Ru06}
B. Green and I. Z. Ruzsa, \emph{Sets with small sumset and rectification},
Bull. Lond. Math. Soc. \textbf{38}(2006), no. 1, 43--52.


\bibitem{Gr-Ta06}B.~Green and T.~Tao, \emph{Compressions, convex geometry and the Freiman-Bilu theorem}, Quart. J. Math. \textbf{57}(2006), 495-504.




\bibitem{Gr13}
D. J. Grynkiewicz, \emph{Structural Additive Theory},
Developments in Mathematics \textbf{30}, Springer, Cham, 2013.

\bibitem{He25}
S.~Hegde,
\emph{An inverse and a stability result for Ruzsa's inequality on triple sumsets},
preprint (2025),
\href{https://arxiv.org/abs/2510.20073}{arXiv:2510.20073}.



\bibitem{Kh92}
A. G. Khovanski\u{\i}, \emph{Newton polyhedron, Hilbert polynomial, and sums of finite sets},
Funct. Anal. Appl. \textbf{26}(1992), no. 4, 276--281.

\bibitem{Kh95}
A. G. Khovanski\u{\i}, \emph{Sums of finite sets, orbits of commutative semigroups, and Hilbert functions},
Funct. Anal. Appl. \textbf{29}(1995), no. 2, 102--112.

\bibitem{Na00}
M. B. Nathanson, \emph{Growth of sumsets in abelian semigroups},
Semigroup Forum \textbf{61}(2000), no. 1, 149--153.

\bibitem{Na96}
M. B. Nathanson, \emph{Additive Number Theory: Inverse Problems and the Geometry of Sumsets},
Grad. Texts in Math. \textbf{165}, Springer, New York, 1996.



\bibitem{Pe11}
G. Petridis, \emph{Plünnecke's inequality},
Combin. Probab. Comput. \textbf{20}(2011), no. 6, 921--938.

\bibitem{Pl70}
H. Plünnecke, \emph{Eine zahlentheoretische Anwendung der Graphtheorie},
J. Reine Angew. Math. \textbf{243}(1970), 171--183.

\bibitem{Ru99}
I. Z. Ruzsa, \emph{An analog of Freiman's theorem in groups},
Ast\'erisque \textbf{258}(1999), 323--326.

\bibitem{Ru94}
I. Z. Ruzsa, \emph{Generalized arithmetic progressions and sumsets},
Acta Math. Hungar. \textbf{65}(1994), no. 4, 379--388.

\bibitem{Sa12}
T. Sanders, \emph{On the Bogolyubov--Ruzsa lemma},
Anal. PDE \textbf{5}(2012), no. 3, 627--655.

\bibitem{Ta-Vu06}
T. Tao and V. Vu, \emph{Additive Combinatorics},
Cambridge Stud. Adv. Math. \textbf{105}, Cambridge Univ. Press, Cambridge, 2006.


\end{thebibliography}
\end{document}